\documentclass[preprint,12pt]{elsarticle}

\usepackage{amssymb}
\usepackage{amsmath}
\usepackage{amsthm}

\journal{Discrete Mathematics}

\newtheorem{theorem}{Theorem}
\newtheorem{lemma}[theorem]{Lemma}

\newtheorem{proposition}[theorem]{Proposition}
\newtheorem{corollary}[theorem]{Corollary}

\newtheorem{problem}{Problem}
\usepackage{xcolor,graphicx,tikz}
\usepackage{algorithm, algorithmicx, algpseudocode}

\begin{document}

\begin{frontmatter}


\title{Zero blocking numbers of graphs with complexity results}
\cortext[corr]{Corresponding author}
\author{Hau-Yi Lin\fnref{nycu}}
\ead{zaq1bgt5cde3mju7@gmail.com}
\author{Wu-Hsiung Lin\corref{corr}\fnref{nycu}}
\ead{wuhsiunglin@nctu.edu.tw}
\author{Gerard Jennhwa Chang\fnref{ntu}}
\ead{gjchang@math.ntu.edu.tw}
\address[nycu]{Department of Applied Mathematics, National Yang Ming Chiao Tung University, Hsinchu 30010, Taiwan}
\address[ntu]{Department of Mathematics, National Taiwan University, Taipei 10617, Taiwan}






\begin{abstract}
For a graph $G$ in which vertices are either black or white,
a {zero forcing process} is an iterative vertex color changing process such that
the only white neighbor of a black vertex becomes black in the next time step.
A {zero forcing set} is an initial subset of black vertices in a zero forcing process
ultimately expands to include all vertices of the graph;
otherwise we call its complement a {zero blocking set}.
The {zero blocking number} $B(G)$ of $G$ is the minimum size of a zero blocking set.
This paper determines zero blocking numbers of the union and the join of two graphs.
It also determines all minimum zero blocking sets of hypercubes.
Finally, a linear-time algorithm for the zero blocking numbers of trees is given.

\end{abstract}




\end{frontmatter}




\section{Introduction}

For a graph $G$ in which vertices are either black or white, a {\em zero forcing process} is an iterative vertex color changing process such that,
if a vertex $w$ is the only white neighbor of a black vertex $b$, then $w$ becomes black in the next time step; in the process we say that $b$ {\em forces} $w$.
A {\em zero forcing set} is an initial subset of black vertices in a zero forcing process ultimately expands to include all vertices of the graph;
otherwise we call it a {\em failed zero forcing set}
and its complement, the initial subset of white vertices, a {\em zero blocking set}.
The {\em zero forcing number} $Z(G)$ of a graph $G$ is the minimum size of a zero forcing set.
The {\em failed zero forcing number} $F(G)$ of $G$ is the maximum size of a failed zero forcing set.
The {\em zero blocking number} $B(G)$ of $G$ is the minimum size of a zero blocking set.
As $F(G)+B(G)=|V(G)|$ for every graph $G$,
studying $F(G)$ and $B(G)$ are equivalent.

In order to check if a set is a failed zero forcing set (also, a zero blocking set)
it is easier to check if it is a {\em final} one,
that is, it is one for which there is no more white vertex can be forced to be black in the zero forcing process.
Being a final failed zero forcing set (also, a final zero blocking set)
is the same as it is a proper (non-empty, respectively) subset of vertices with the condition
that there exists no black vertex adjacent to exactly one white vertex. 

Notice that if an initial set of black vertices be expanded to include all vertices of the graph
through {\em some} vertex forcing sequence, then it can be done through {\em any} possible vertex forcing sequence.
It is also the case that starting from a failed zero forcing set,
the zero forcing process will terminate at a unique final failed zero forcing set
no matter which forcing vertex sequence is chosen.
This is also true for a zero blocking set, see \cite{2008a}.
Final failed zero forcing sets were also called ``stalled sets'' \cite{2015fjs,2017s}
and final zero blocking sets were also called ``forts'' \cite{2017f}.
A maximum failed zero forcing set is a final failed zero forcing set
and a minimum zero blocking set is a final zero blocking set,
but the converses are not true.
Also, $F(G)$ is the maximum size of a final failed zero forcing set
and $B(G)$ is the minimum size of a final zero blocking set.

The zero forcing process was introduced in \cite{2008AIM} to study minimum rank problems in linear algebra,
and independently in \cite{2007bg} to explore quantum systems memory transfers in quantum physics.
The computation of the zero forcing number of a graph is NP-hard \cite{2008a}.
Considerable effort has been made to find exact values and bounds of this
number for specific classes of graphs, and to investigate a variety of related concepts arising in zero forcing processes.

The concept of failed zero forcing number was first introduced in \cite{2015fjs}.
The computation of the failed zero forcing number is NP-hard \cite{2017s}.
Results for exact values and bounds of this number were also established in
{\cite{2024agm,2016ajps,2020bccknt,2024c,2017f,2021grtn,
2024grtn,2020ksv,2023ktn,2025LLc,2020paA,2020paB,2023su}.
This paper determines zero blocking numbers of unions and joins of graphs.
Extending the result in \cite{2024agm} and \cite{2024c}, it determines all minimum zero blocking sets of hypercubes.
Finally, a linear-time algorithm for the zero blocking numbers of trees is given.

\section{Preliminary}

All graphs in this paper are finite, undirected, without loops and multiple edges.
In a graph, the {\em neighborhood} of a vertex $x$ is the set $N(x)$ of all vertices adjacent to $x$ and the {\em closed neighborhood} of $x$ is $N[x]=\{x\}\cup N(x)$.
A vertex $x$ is {\em isolated} if $N(x)=\emptyset$.
A vertex $x$ is a {\em leaf} if $|N(x)|=1$.
Two distinct vertices $x$ and $y$ are {\em twins} if $N(x)=N(y)$ or $N[x]=N[y]$.

The {\em union} of two graphs $G$ and $H$ is the graph $G \cup H$ with
vertex set $V(G\cup H)=V(G)\cup V(H)$ and
edge set $E(G\cup H)=E(G)\cup E(H)$.
The {\em join} of two graphs $G$ and $H$ is the graph $G + H$ with
vertex set $V(G+H)=V(G)\cup V(H)$ and
edge set $E(G+H)=E(G)\cup E(H)\cup\{xy: x\in V(G), y\in V(H)\}$.
The {\em Cartesian product} of two graphs $G$ and $H$ is the graph $G \Box H$ with
vertex set $V(G\Box H)=\{(x,y): x\in V(G), y\in V(H)\}$ and
edge set $E(G\Box H)=\{(x,y) (x',y'): (x=x', yy'\in E(H))$ or $(xx'\in E(G), y=y')\}$.

In a graph $G$, a {\em stable set} is a subset of vertices
in which no two vertices are adjacent.
The {\em stability number} $\alpha(G)$ of $G$
is the maximum size of a stable set in $G$.
For instance,
$\alpha(P_n)=\lceil \frac{n}{2} \rceil$ for $n\ge 1$ and
$\alpha(C_n)=\lfloor \frac{n}{2} \rfloor$ for $n\ge 3$.
A {\em vertex cover} is a subset $C$ of vertices
such that every edge is incident to at least one vertex in $C$.
A {\em dominating set} of $G$ is a subset $D$ of vertices
such that every vertex in $G$ is either in $D$
or adjacent to some vertex in $D$.
The {\em domination number} $\gamma(G)$ of $G$
is the minimum size of a dominating set of $G$.
For instance,
$\gamma(P_n)=\lceil \frac{n}{3} \rceil$ for $n\ge 1$ and
$\gamma(C_n)=\lceil \frac{n}{3} \rceil$ for $n\ge 3$.

The following properties are easy to see.

\begin{proposition}
If a graph $G$ has an isolated vertex, then $B(G)=1$, else $B(G)\ge 2$.
\end{proposition}

\begin{proposition} \label{component}
If a graph\ $G$ has $r$ components $G_1, G_2,\ldots,G_r$, then $B(G)=\min\limits_{1\le i\le r} B(G_i)$.
\end{proposition}

Suppose $G$ has no isolated vertex.
If $G$ has a pair of twins $x$ and $y$, then $\{x,y\}$ is a minimum zero blocking set.
On the other hand, if $\{x,y\}$ is a minimum zero blocking set, then $x$ and $y$ is a pair of twins,
for otherwise there is some black vertex $z$ adjacent to exact one of these two vertices, a contradiction.

\begin{theorem}[\cite{2015fjs}] \label{twins}
For a graph $G$ without isolated vertices, $B(G)=2$ if and only if $G$ has a pair of twins.
\end{theorem}

In the path $P_n$ or cycle $C_n$, $W$ is a minimum zero blocking set if and only if all vertices not in $W$ form a maximum stable set containing no leaves.
As such a stable set is of size $\lceil\frac{n-2}{2}\rceil$ for $P_n$ and $\lfloor\frac{n}{2}\rfloor$ for $C_n$, we have
$B(P_n) = n - \lceil\frac{n-2}{2}\rceil = \lceil\frac{n+1}{2}\rceil$ and
$B(C_n) = n- \lfloor\frac{n}{2}\rfloor = \lceil\frac{n}{2}\rceil$.

\begin{theorem}[\cite{2015fjs}] \label{path-cycle}
$B(P_n)=\lceil\frac{n+1}{2}\rceil$ for $n\ge 1$.
$B(C_n)=\lceil\frac{n}{2}\rceil$ for $n\ge 3$.
\end{theorem}

\section{Zero blocking numbers of unions and joins} \label{sec-union-join}

The class of $P_4$-free graphs plays an important role in graph theory.
It is well-known that they are the same as cographs.
A {\em complement reducible graph}, called a {\em cograph} for short,
is defined recursively as follows.
\begin{itemize}
\item[(C1)] A graph of a single vertex is a cograph.
\item[(CU)] If $G_1$ and $G_2$ are cographs, then so is their union $G_1\cup G_2$.
\item[(CJ)] If $G_1$ and $G_2$ are cographs, then so is their join $G_1+G_2$.
\end{itemize}
Cographs have the following property, which gives the zero blocking numbers of them by Theorem \ref{twins}.

\begin{theorem}[\cite{1981cLs}]
Every cograph of at least two vertices has a pair of twins.
\end{theorem}

\begin{corollary} \label{cograph}
If $G$ is a cograph,  then $B(G)=1$ if $G$ has an isolated vertex, else $B(G)=2$.
\end{corollary}

The following gives zero blocking numbers of the union and the join of two graphs.
These also gives the result for cographs, but more other graphs.
For this purpose, a variation of zero blocking number is defined as follows.
The {\em second zero blocking number} $B'(G)$ of a graph $G$ is the minimum size of a zero blocking set of size at least two.
Notice that if $G$ has exactly one vertex, then $B'(G)=\infty$, else $B'(G)$ is finite.

\begin{proposition}\label{2nd}
For a graph $G$, we have 
$$
B'(G)=\left\{\begin{array}{ll}
         2,      & \mbox{if $G$ has at least two isolated vertices}; \\
         B(G-x), & \mbox{if $G$ has exactly one isolated vertex $x$}; \\
         B(G),   & \mbox{if $G$ has no isolated vertex}.
             \end{array}
       \right.
$$ 
\end{proposition}

\begin{proposition}
$B(G\cup H)=\min\{B(G),B(H)\}$ for all graphs $G$ and $H$.
\end{proposition}

\begin{proposition}
If $G\cup H$ has at least two isolated vertices, 
then $B'(G\cup H)=2$, else $B'(G\cup H)=\min\{B'(G),B'(H)\}$.
\end{proposition}

\begin{theorem} \label{join}
If $G$ and $H$ are graphs of $n$ and $m$ vertices respectively, then
$B'(G+H)=B(G+H)=\min\{B'(G),B'(H),a\}$, where
$$
a=\left\{\begin{array}{ll}
 \gamma(G)+\gamma(H), & \mbox{if $n=1$ or $m=1$ or $\gamma(G)=\gamma(H)=1$}; \\
 3,                   & \mbox{if $n\ge 2$, $m\ge 2$ and}\\ 
                      & \mbox{$(\gamma(G)\le 2$ or $\gamma(H)\le 2$ but not $\gamma(G)=\gamma(H)=1)$}; \\
 4,                   & \mbox{if $\gamma(G)\ge 3$ and $\gamma(H)\ge 3$}.
  \end{array}\right.
$$
\end{theorem}
\begin{proof}
Since $G+H$ has no isolated vertex, $B'(G+H)=B(G+H)$ by Proposition \ref{2nd}.

First, prove that $B(G+H)\le \min\{B'(G),B'(H),a\}$.
Since a minimum second zero blocking set of $G$ is
a zero blocking set of $G+H$, we have $B(G+H)\le B'(G)$.
Similarly, $B(G+H)\le B'(H)$.
Now we construct a zero blocking set of $G+H$ for each condition of $a$ in three cases.

Case 1. $n=1$ or $m=1$ or $\gamma(G)=\gamma(H)=1$.
In this case we choose
a minimum dominating set $D$ (respectively, $D'$) of $G$ (respectively, $H$).
Then $D\cup D'$ is a zero blocking set of $G+H$,
and so $B(G+H)\le |D\cup D'|= \gamma(G)+\gamma(H)=a$, 
the last equality follows from the definition if $a$.

Case 2. $n\ge 2$, $m\ge 2$,
($\gamma(G)\le 2$ or $\gamma(H)\le 2$ but not $\gamma(G)=\gamma(H)=1$).
In this case, $a=3$.
By symmetry, assume $\gamma(G)\le 2$.
As $n\ge 2$, we may choose
a dominating set $D$ of $G$ with $|D|=2$ and a vertex $y\in V(H)$.
Then $D\cup\{y\}$ is a zero blocking set of $G+H$,
and so $B(G+H)\le |D\cup\{y\}| = 3=a$.

Case 3. $\gamma(G)\ge 3$ and $\gamma(H)\ge 3$.
In this case, $a=4$. As $n\ge 3$ and $m\ge 3$, we may choose vertices $x,y\in V(G)$ and $z,w\in V(H)$.
Then $\{x,y,z,w\}$ is a zero blocking set of $G+H$, and so $B(G+H)\le |\{x,y,z,w\}|=4=a$.

In summary, $B(G+H)\le \min\{B'(G),B'(H),a\}$.

To show $B(G+H)\ge \min\{B'(G),B'(H),a\}$,
we choose a minimum zero blocking set $W$ of $G+H$.
Let $D=W\cap V(G)$ and $D'=W\cap V(H)$.
If $|D'|=0$, then every vertex in $V(H)$ is black and so has at least two white neighbors, which are in $V(G)$.
Also, $W=D$ is a zero blocking set of $G$ and so is a second zero blocking set of $G$.
Hence $ B(G+H) =  |W|\ge B'(G)$.
Similarly, if $|D|=0$, then $B(G+H) =  |W|\ge B'(H)$.

Now suppose $|D|\ge 1$ and $|D'|\ge 1$.
We consider the following cases. First, claim that if $|D|=1$ then $D'$ is a dominating set of $H$,
for otherwise $V(H)$ has a black vertex adjacent to the only white vertex in $D$,
a contradiction.
Similarly, if $|D'|=1$, then $D$ is dominating set of $G$.

Case 1.
If $n=1$, then $|D|=1$ and so $D$ (respectively, $D'$) is a dominating set of $G$ (respectively, $H$ by the above claim).
In this case, $ B(G+H) =  |W|=|D|+|D'|\ge \gamma(G)+\gamma(H)=a$.
Similarly, if $m=1$, then $ B(G+H) = |W|\ge a$.

Case 2.
If $n\geq 2$ and $m\ge 2$,
then $a=2,3$ or $4$ depending on the values of $\gamma(G)$ and $\gamma(H)$.

Case 2.1.
If $|W|=2$, then $|D|=|D'|=1$ and so $D$ (respectively, $D'$) is a dominating set of $G$ (respectively, $H$) by the above claim.
In this case, $\gamma(G)=\gamma(H)=1$, implying $a=\gamma(G)+\gamma(H)=2$.
Thus $ B(G+H) = |W|=a$.

Case 2.2.
If $|W|=3$, then either $|D|=1$ or $|D'|=1$.
By symmetry, assume that $|D|=1$ and so $|D'|=2$.
By the above claim, $D'$ is a dominating set of $H$
and so $\gamma(H)\le 2$, implying $a=2$ or $3$.
Thus $B(G+H) = |W|\ge a$.

Case 2.3.
Otherwise, $B(G+H) = |W|\ge 4\ge a$.
 
In summary, $B(G+H)\ge \min\{B'(G),B'(H),a\}$,
and theorem follows.
\end{proof}

Corollary \ref{cograph} also follows from the above theorem by induction.
Since wheel $W_n=K_1+C_n$, the following result also follows from the fact that
$B'(C_n)=B(C_n)=\lceil\frac{n}{2}\rceil$ and $\gamma(C_n)=\lceil\frac{n}{3}\rceil$.

\begin{theorem} [\cite{2015fjs}]
If $n\ge 3$, then $B(W_n)=\lceil\frac{n+3}{3}\rceil$ except $B(W_4)=2$.
\end{theorem}

Similar values are available for the fan $mK_1+P_n$ and the cone $mK_1+C_n$.
If $m=1$, then $B(K_1+P_n)=\lceil\frac{n+3}{3}\rceil$;
and if $m\ge 2$, then $B(mK_1+P_n)=B(mK_1+C_n)=2$, since there are twins.

\section{Minimum zero blocking sets of hypercubes}  \label{sec-cube}

The {\em $n$-dimensional hypercube} is the graph $Q_n:=\overbrace{P_2\Box P_2\Box\cdots\Box P_2}^n$.
For every two sets $A$ and $B$,
denote $A \Delta B = (A\cup B)\backslash (A \cap B)$.
It is easy to see that $A\Delta\emptyset=A$ and $A\Delta A=\emptyset$ for every set $A$.
Using the set notation, we may describe $Q_n$ as the graph with
vertex set $V(Q_n)=\{x: x\subseteq [n]\}$ and
edge set $E(Q_n) =\{xy: x\subseteq y \subseteq [n],\, |y\backslash x|=1\}
=\{xy: x,y\subseteq [n],\, |x\Delta y|=1\}$,
where $[n]=\{1,2,\ldots,n\}$.
Notice that $Q_n$ is vertex transitive, that is, for every two vertices $x$ and $y$ there is an automorphism $f$ on $Q_n$ such that $f(x)=y$.
This is true as the function $f:V(Q_n)\to V(Q_n)$ defined by $f(v)=v\Delta x\Delta y$
satisfies that $f(u)\Delta f(v)=(u\Delta x\Delta y)\Delta(v\Delta x\Delta y)=u\Delta v$ for every two vertices $u$ and $v$.

It was proved in \cite{2024agm} and \cite{2024c} that $B(Q_n)=n$ for $n\ge 2$.
The following theorem extends their result to characterize minimum zero blocking sets $W$ of $Q_n$ with $|W|=n$.

\begin{theorem}
If $n\geq 2$, then $B(Q_n)=n$.
Furthermore, a zero blocking set $W$ of $Q_n$ is minimum
if and only if either $W=N(x)$ for some vertex $x$
or else $n=4$ with $W=N(y)\Delta N(z)$ for two vertices $y$ and $z$ with $|y\Delta z|=2$.
\end{theorem}
\begin{proof}
Choose $W=N(\emptyset)=\{x: |x|=1\}$ as the set of white vertices.
Consider all black vertices $y$ in three cases.
If $y=\emptyset$, then $y$ is adjacent to white vertices $\{1\}$ and $\{2\}$.
If $y=\{i,j\}$, then $y$ is adjacent to two white vertices $\{i\}$ and $\{j\}$.
If $|y|=3$, then $y$ is adjacent to no white vertex.
Hence $W$ is a zero blocking set and so $B(Q_n)\le |W|=n$.

On the other hand, suppose $W$ is a minimum zero blocking set of $Q_n$.
Since $n< 2^n$, there is some black vertex, say $\emptyset$ by symmetry,
adjacent to some and so exactly $r\ge 2$ white vertices, say $\{1\}, \{2\}, \ldots, \{r\}$ by symmetry.
Consider the set $S=\{\{i,j\}:1\le i\leq r<j\le n\}$ of size $r(n-r)$,
in which each vertex $\{i,j\}$ is adjacent to one white vertex $\{i\}$ and one black vertex $\{j\}$.
If $\{i,j\}\in S\backslash W$, then it is adjacent to some white vertex $\{i,j,k\}$.
Let $T$ be the set of all such $\{i,j,k\}$.
Notice that each $\{i,j,k\}\in T$ is adjacent to at most two vertices in $S\backslash W$, since $i\le r<j$ implies that $\{i,j,k\}$ has only two neighbors in $S$.
Hence $|T|\ge |S\backslash W|/2$.
Then
$|S\cap W|+|T|\ge |S\cap W|+|S\backslash W|/2\ge |S|/2=r(n-r)/2$.
And, if the equalities hold, then $|S\cap W|=0$ and each $\{i,j,k\}\in T$ has two black neighbors in $S$;
i.e., $\{i,j\}$ and $\{k,j\}$ when $k\le r$, and $\{i,j\}$ and $\{i,k\}$ when $k>r$.
Therefore,
$$
 |W|\ge |N(\emptyset)\cap W|+|S\cap W|+ |T|
    \ge r+r(n-r)/2=n+(r-2)(n-r)/2\ge n.
$$
Hence $B(Q_n)=|W|\ge n$ and so $B(Q_n)=n$.

Next to see the second part of the theorem.
If $W=N(x)$, then it is a minimum zero blocking set by using the same arguments for $N(\emptyset)$.
If $n=4$ and $|y\Delta z|=2$, then $|N(y)\cap N(z)|=2$ and $|N(y)|=|N(z)|=|N(y)\Delta N(z)|=4$.
It is not hard to check that $N(y)\Delta N(z)$ is a zero blocking set as desired.

On the other hand, suppose $W$ is a minimum zero blocking set, as in the second paragraph.
Then $|S\cap W|+|T|=r(n-r)/2$ and $(r-2)(n-r)/2=0$.
Either $r=2$ or $n=r$.
The case of $n=r$ or $n=2$ gives that $W=N(\emptyset)$.
The case of $r=2$ with $n\geq 3$ gives that $|S\cap W|=0$ and $|T|=2(n-2)/2=n-2$,
$W=\{\{1\},\{2\}\}\cup T,  S=\{\{1,j\},\{2,j\}:3\leq j\leq n\}$ and
$T\subseteq \{\{1,2,j\},\{1,j,k\},\{2,j,k\}: 3\le j\le n, 3\le k\le n\}$
in which each vertex in $S$ is adjacent to exactly one vertex in $T$.

Case 1. $n=3$.
In this case we have $T=\{\{1,2,3\}\}$ and so $W=N(\{1,2\})$.

Case 2. $n=4$.
In this case we have $T=\{\{1,2,3\},\{1,2,4\}\}$ or $T=\{\{1,3,4\}, \{2,3,4\}\}$.
For the former case, $W=N(\{1,2\})$.
For the latter case, $W=N(y)\Delta N(z)$ 
where $y=\{1,2\}$ and $z=[4]$ with $|y\Delta z|=|\{3,4\}|=2$.

Case 3. $n\geq 5$.
In this case, if $\{i,j,k\}\in T$ for some $1\le i\le 2 <j<k\leq n$,
then for $\ell>2$ other than $j$ and $k$
we have $\{i,j,\ell\}, \{i,k,\ell\}, \{j,k,\ell\}\notin T$,
and so the black vertex $\{i,j,k,\ell\}$ is adjacent to exactly one white vertex $\{i,j,k\}$,
a contradiction.
Hence, $T=\{\{1,2,j\}: 2<j\le n\}$, so $W=N(\{1,2\})$.

The second part of the theorem then follows.
\end{proof}

\section{Complexity results} \label{sec-tree}
 
This section considers complexity results for zero blocking numbers of graphs.

\begin{problem}
{\bf The (Final) Zero Blocking Problem}\\
{\rm Given:} A graph $G$ and an integer $k$.\\
{\rm Question:} Does $G$ contain a (final) zero blocking set of size at most $k$?
\end{problem}

As a zero blocking set is a superset of some final zero blocking set,
these two versions are essentially the same.
Shitov \cite{2017s} proved the complement version of the problem is NP-complete.

\begin{problem}
{\bf The (Final) Failed Zero Forcing Problem}\\
{\rm Given:} A graph $G$ and an integer $k$.\\
{\rm Question:} Does $G$ contain a (final) failed zero forcing set of size at least $k$?
\end{problem}

\begin{theorem}[\cite{2017s}]
The {final} failed zero forcing problem is NP-complete.
\end{theorem}

In the following we prove that the problem is NP-complete even for bipartite graphs and for chordal graphs.
Recall that a graph is {\em chordal} if every cycle of length at least four has a chord.
Our proof is by transforming the vertex cover problem to it.

\begin{problem}
{\bf The Vertex Cover Problem}\\
{\rm Given:} A graph $G$ and an integer $k$.\\
{\rm Question:} Does $G$ contain a vertex cover of size at most $k$?
\end{problem}

Notice that we may assume that $G$ is connected, since adding a new vertex adjacent to all old vertices increases the minimum size of a vertex cover by exactly one.
We may also assume that $k\le |V(G)|-2$,
since if $G=K_n$ then the minimum size of a vertex cover is equal to $|V(G)|-1$
otherwise it is less than $|V(G)|-1$,
meanwhile $B(K_n)=2$.

\begin{theorem}
The {final} zero blocking problem is NP-complete for bipartite graphs and for chordal graphs.
\end{theorem}

\begin{proof}
Given a connected but non-complete graph $G$ of $n$ vertices and an integer $k\le n-2$,
we construct a bipartite graph $G'$ and prove that ``$G$ has a vertex cover of size at most $k$''
if and only if ``$G'$ has a {final} zero blocking set of size at most $k+1$''.
The theorem then follows from the fact that the vertex cover problem is NP-complete.
From $G$, let $G'$ be the graph with
vertex set $V(G')=\{s\} \cup V(G) \cup (\cup_{e\in E(G)} V_e)$ where
           $V_e=\{e_0, e_1, \ldots,e_{2n}\}$ and
edge set $E(G')=\{s e_0: e\in E(G)\} \cup \{xe_0,ye_0: e=xy\in E(G)\} \cup
           (\cup_{e\in E(G)} E_e)$,
where $E_e=\{e_0 e_1, e_1 e_2, \ldots, e_{2n-1} e_{2n}\}$.
Notice that $G'$ is a bipartite graph.

If $C$ is a vertex cover in $G$ of size at most $k$,
then $W=\{s\}\cup C$ is a {final} zero blocking set of $G'$ of size at most $k+1$
by checking the following cases:
a black vertex $x\in V(G)\backslash W$ is only adjacent to black vertex $e_0$ with $x\in e$,
a black vertex $e_i\in V_e\backslash \{e_0\}$ is only adjacent to black vertices in $V_e$, and a black vertex $e_0=xy$ is adjacent to white vertex $s$ and one white vertex in $\{x,y\}$, since $C$ is a vertex cover.

On the other hand, suppose $W$ is a {final} zero blocking set of $G'$ of size at most $k+1<n$.
First, every $V_e$ contains two black vertices $e_i$ and $e_{i+1}$ for otherwise it contains at least $n$ white vertices, impossible.
Consequently, all vertices of $V_e$ are black for all $e\in E(G)$.
Next, $s$ is white, for otherwise suppose $s$ is black.
As $|W|<n$, $V(G)$ has at least one black vertex $x$,
and for any neighbor $y$ of $x$,
the edge $xy=e$ causes that the black vertex $e_0$ implies that $y$ is also black.
Continue this process, since $G$ is connected, all vertices in $V(G)$ are black, impossible.
Now $s$ is white.
Then for every edge $e=xy$ in $G$, the black vertex $e_0$ has a white neighbor $s$ and hence has another white neighbor which is either $x$ or $y$.
Then $C=W\backslash \{s\}$ is a vertex cover of $G$ of size at most $k$, as desired.

Since the vertex cover problem is NP-complete,
the {final} zero blocking problem is NP-complete for bipartite graphs.

If we make the set $\{e_0: e\in E(G)\}$ a clique, then the graph $G'$ is chordal and the arguments above are precisely the same.
Hence, the {final} zero blocking problem is NP-complete for chordal graphs.
\end{proof}

In the following, we design a linear-time algorithm for finding $B(T)$ of a tree $T$
by using a dynamic programming approach,
which is a power method for solving many discrete optimization problems,
see \cite{1962bd,2025c,1977dL,1966n} for references.

Now a specific vertex $v$ is chosen from $G$.
One may consider $(G,v)$ as a graph rooted at the chosen vertex $v$.
A minimum zero blocking set $W$ of $(G,v)$ either contains $v$ or does not.
So it is useful to consider the following two parameters, which are the ordinary zero blocking number $B(G)$ with boundary conditions.
\begin{eqnarray*}
 B^1(G,v) &:=& \min\{|W|: \mbox{$v\notin W$ and $W$ is a {final} zero blocking set of $G$} \}. \\
 B^0(G,v) &:=& \min\{|W|: \mbox{$v\in W$ and $W$ is a {final} zero blocking set of $G$} \}.
\end{eqnarray*}
Notice that $\min \emptyset=\infty$.
For instance, if $G=P_n$ and $v$ is a leaf, then $B^1(G,v)=\infty$.

It is clear to have the following equality, since a {final} zero blocking set of $(G,v)$ either contain $v$ or does not.

\begin{lemma} \label{zbs1}
$B(G,v) = \min\{B^1(G,v)$, $B^0(G,v)\}$ for every rooted graph $(G,v)$.
\end{lemma}

The dynamic programming approach for the zero blocking set problem in trees needs the following graph operation.
The {\em composition} of two rooted graphs $(G,v)$ and $(H,u)$ is the rooted graph $(I,v)$   \index{composition}
obtained from the disjoint union of $(G,v)$ and $(H,u)$ by joining a new edge $vu$, see Figure \ref{s4f4} for an example.

\begin{figure}[htb]
\centering
\begin{tikzpicture}
\draw[very thick] (0,0) --(-1,-3)--(1,-3)--(0,0)--
(3,-1) --+(-1,-3)--+(1,-3)--+(0,0);
\draw[very thick,fill]
(0,0) circle (.1) node [above] {$v$}
(3,-1) circle (.1)  node [above] {$u$}
(0,-2) node {$G$}(3,-3) node {$H$};
\end{tikzpicture} 
\caption{The composition of two rooted graphs $(G,v)$ and $(H,u)$.} \label{s4f4}
\end{figure}
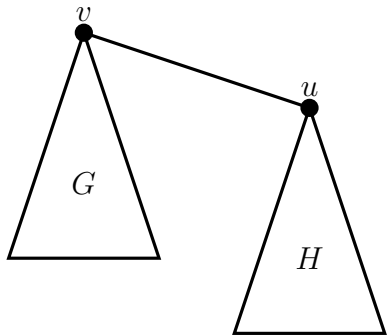

As a tree of $n$ vertices can be obtained from $n$ one-vertex graphs by a sequence of $n-1$ composition, to find the zero blocking  number of a tree,
it is useful to use $B^1(G,v)$, $B^0(G,v)$, $B^1(H,u)$ and $B^0(H,u)$ to compute $B^1(I,v)$ and $B^0(I,v)$.
During the derivation, there is a problem as follows.

Suppose $W$ is a {final} zero blocking set of $(I,v)$ with $v \notin W$, i.e., $v$ is black.
There are two cases.
For the case of $u \notin W$, the set $W\cap V(G)$ either is empty or is a {final} zero blocking set of $G$.
For the case of $u\in W$, the black vertex $v$ must have at least one white neighbor in $V(G)$, but may be just one.
So $W\cap V(G)$ may not be a {final} zero blocking set of $G$.
In order to resolve this difficulty, the following new parameter is introduced.
For a rooted graph $(G,v)$, a {\em {final} $v$-zero blocking set} is a subset $W\subseteq V(G)$ of white vertices such that $v$ is black, $v$ has at least one (but may be just one) white neighbor and every other black vertex can't have exactly one white neighbor.
\begin{eqnarray*}
 B^{11}(G,v) &:=& \min\{|W|: \mbox{$v\notin W$ and $W$ is a {final} $v$-zero blocking set of $G$} \}.
\end{eqnarray*}

The dynamic programming solution for the zero blocking number of a tree is based on the following theorem.

\begin{theorem} \label{zbs2}
For the composition $(I,v)$ of $(G,v)$ and $(H,u)$, the following hold.
\begin{enumerate}
\item[$(1)$] $B^1(I,v) = \min\{B^1(G,v), B^1(H,u), B^{11}(G,v) + B^0(H,u)\}$.
\item[$(2)$] $B^{11}(I,v)  = \min\{B^{11}(G,v), B^0(H,u)\}$.
\item[$(3)$] $B^0(I,v) = \min\{B^0(G,v) + B^{11}(H,u), B^0(G,v) + B^0(H,u)\}$.
\end{enumerate}
\end{theorem}
\begin{proof}
(1) This follows from that for $v\notin W$, $W$ is a {final} zero blocking set of $I$
if and only if either ($W$ is a {final} zero blocking set of $G$ or $H$ when $u\notin W$)
or else ($W\cap V(G)$ is a {final} $v$-zero blocking set of $G$ and $W\cap V(H)$ is a {final} zero blocking set of $H$ when $u\in W$).

(2) This follows from that for $v\notin W$, $W$ is a {final} $v$-zero blocking set of $I$
if and only if either ($W$ is a {final} $v$-zero blocking set of $G$ when $u\notin W$)
or else ($W$ is a {final} zero blocking set of $H$ when $u\in W$).

(3) This follows from that for $v\in W$, $W$ is a {final} zero blocking set of $I$
if and only if either ($W\cap V(G)$ is a {final} zero blocking set of $G$ and $W\cap V(H)$ is a {final} $u$-zero blocking set of $H$ when $u\notin W$)
or else ($W\cap V(G)$ is a {final} zero blocking set of $G$ and $W\cap V(H)$ is a {final} zero blocking set of $H$ when $u\in W$).
\end{proof}

Lemma \ref{zbs1} and Theorem \ref{zbs2} then give the following linear-time algorithm
for the zero blocking set problem in trees.

\begin{algorithm}[htp!]
\caption{\bf zbsTreeD.}\label{algo1}
\begin{algorithmic}[xxx]
  \Require A tree $T=(V,E)$.
  \Ensure $B(T)$.
\State
Choose a vertex $v_n$ and order $V$ into $v_1,v_2,\ldots, v_n$
such that $d(v_n,v_i)<d(v_n,v_j)$ implies $i>j$.
For $1\le i<n$, let $p_i$ be the unique neighbor of $v_i$ nearer to $v_n$ than $v_i$;
\For {$i=1$ to $n$}
\State $B^1(v_i)\,\, \leftarrow \infty$;
\State $B^{11}(v_i) \leftarrow \infty$;
\State $B^0(v_i)\,\, \leftarrow 1$;
\EndFor
\For {$i=1$ to $n-1$}
\State  $B^1(p_i)\,\, \leftarrow \min\{B^1(p_i),B^1(v_i),B^{11}(p_i)+B^0(v_i)\}$;
\State  $B^{11}(p_i) \leftarrow \min\{B^{11}(p_i),B^0(v_i)\}$;
\State  $B^{0}(p_i)\,\, \leftarrow \min\{B^0(p_i) + B^{11}(v_i),B^0(p_i) + B^0(v_i)\}$;
\EndFor
\State $B(T) \leftarrow \min\{B^1(v_n), B^0(v_n)\}$;
\end{algorithmic}
\end{algorithm}

The above method can be generalized to {\em block} graphs, which are graphs whose blocks are complete graphs.

It is expected that polynomial-time algorithms for the problem in interval graphs and strongly chordal graphs exist.


\end{document}